\documentclass[12pt]{article}
\usepackage[top=2.5cm,bottom=2.5cm,left=2.9cm,right=2.9cm]{geometry}
\usepackage{amssymb}
\usepackage{amsmath,amsthm}
\usepackage{graphicx}
\usepackage{enumerate}
\usepackage{tikz}
\usepackage{mathrsfs}
\usepackage{verbatim}

\usepackage{colortbl}

\newtheorem{remark}{Remark}[section]

\newtheorem{theorem}[remark]{Theorem}
\newtheorem{proposition}[remark]{Proposition}

\newtheorem{problem}[remark]{Problem}

\title{A note on the $k$-tuple domination number of graphs}

\author{Abel Cabrera Mart\'inez\\
{\small Universitat Rovira i Virgili, Departament d'Enginyeria Inform\`atica i Matem\`atiques } \\  
{\small Av. Pa\"{\i}sos
Catalans 26, 43007 Tarragona, Spain.} \\
{\small
  abel.cabrera\@@urv.cat}
}

\begin{document}
\maketitle

\begin{abstract}
In a graph $G$, a vertex dominates itself and its neighbours. A set $D\subseteq V(G)$ is said to be a $k$-tuple dominating set of $G$  if $D$ dominates every vertex of $G$ at least $k$ times.  The minimum cardinality among all $k$-tuple dominating sets  is the $k$-tuple domination number of $G$. In this note, we provide new bounds on this parameter. Some of these bounds generalize other ones that have been given for the case $k=2$.
\end{abstract}

\noindent {\it Keywords}:
$k$-domination; $k$-tuple domination.

\section{Introduction}  
Throughout this note we consider simple graphs $G$ with vertex set $V(G)$. Given a vertex $v\in V(G)$, $N(v)$ denotes the \emph{open neighbourhood} of $v$ in $G$. In addition, for any set $D\subseteq V(G)$, the \emph{degree} of $v$ in $D$, denoted by $\deg_D(v)$, is the number of vertices in $D$ adjacent to $v$, i.e., $\deg_D(v)=|N(v)\cap D|$. The \emph{minimum} and \emph{maximum degrees} of $G$ will be denoted by $\delta(G)$ and $\Delta(G)$, respectively. Other definitions not given here can be found in standard graph theory books such as \cite{West1996}. 

Domination theory in graphs have been extensively studied in the literature. For instance, see the books \cite{Book-TopicsDom-2020,hhs1,hhs2}. 
A set $D\subseteq V(G)$ is said to be a \emph{dominating set} of $G$ if $\deg_D(v)\geq 1$ for every $v\in V(G)\setminus D$. The \emph{domination number} of $G$ is the minimum cardinality among all dominating sets of $G$ and it is denoted by $\gamma(G)$. We define a  $\gamma(G)$-set as a dominating set of cardinality $\gamma(G)$. The same agreement will be assumed for optimal parameters associated to other characteristic sets defined in the paper.

In 1985, Fink and Jacobson \cite{MR812671,MR812672}  extended  the idea of domination in graphs to the more general notion of $k$-domination.
A set $D\subseteq V(G)$ is said to be a $k$-\emph{dominating set} of $G$ if $\deg_D(v)\geq k$ for every  $v\in V(G)\setminus D$. The $k$-\emph{domination number} of $G$, denoted by $\gamma_k(G)$, is the minimum cardinality among all $k$-dominating sets of $G$. 
Subsequently, and as expected, several variants for $k$-domination were introduced and studied by the scientific community. In two different papers published in 1996 and 2000, Harary and Haynes \cite{MR1401362,Harary2000} introduced the concept of double domination and, more generally, the concept of $k$-tuple domination.
Given a graph $G$ and a positive integer $k\leq \delta(G)+1$, a $k$-dominating set $D$ is said to be a $k$-\emph{tuple dominating set} of $G$ if $\deg_D(v)\geq k-1$ for every $v\in D$. 
The $k$-\emph{tuple domination number} of $G$, denoted by  $\gamma_{\times k}(G)$, is the minimum cardinality among all $k$-tuple dominating sets of $G$. The case $k=2$ corresponds to double domination, in such a case, $\gamma_{\times 2}(G)$ denotes the double domination number of graph $G$.

In this note, we provide new bounds on the $k$-tuple domination number. Some of these bounds generalize other ones that have been given for the double domination number.

\section{New bounds on the $k$-tuple domination number}

Recently, Hansberg and Volkmann \cite{HansVolk2020}  put into context all relevant research results on multiple domination that have been found up to 2020. In that chapter, they posed the following open problem.

\begin{problem}\label{prob-1}(Problem $5.8$, p.$194$, {\rm \cite{HansVolk2020})}
 Give an upper bound for $\gamma_{\times k}(G)$ in terms of $\gamma_k(G)$ for any graph $G$ of minimum degree $\delta(G)\geq k-1$.
\end{problem}

A fairly simple solution for the problem above is given by the straightforward relationship $\gamma_{\times k}(G)\leq k\gamma_{k}(G)$, which can be derived directly by constructing a set of vertices $D'\subseteq V(G)$ of minimum cardinality from a $\gamma_k(G)$-set $D$ such that $D\subseteq D'$ and $\deg_{D'}(x)\geq k-1$ for every vertex $x \in D$. From this construction above, it is easy to check that $D'$ is a $k$-tuple dominating set of $G$ and so,
$$\gamma_{\times k}(G)\leq |D'|=|D|+|D'\setminus D|\leq |D|+(k-1)|D|= k\gamma_{k}(G).$$
This previous inequality was surely considered by Hansberg and Volkmann and, in that sense, they have established the previous problem assuming that $\gamma_{\times k}(G)<k\gamma_{k}(G)$ for every graph $G$ with $\delta(G)\geq k-1$.

We next confirm their suspicions and provide a solution to Problem \ref{prob-1}.

\begin{theorem}\label{teo-open-problem}
Let $k\geq 2$ be an integer. For any graph $G$ with $\delta(G)\geq k-1$,
$$\gamma_{\times k}(G)\leq k\gamma_{k}(G)-(k-1)^2.$$
\end{theorem}

\begin{proof}
Let $D$ be a $\gamma_k(G)$-set. As $\gamma_{\times k}(G)\leq |V(G)|$ we assume, without loss of generality, that $k|D|-(k-1)^2\leq |V(G)|$. 
Now, let $U=\{u_1, \ldots, u_{k-1}\}\subseteq V(G)\setminus D$,  $D'=D\cup U$ and $D_0=\{v\in D :  \deg_{D'}(v)<k-1\}$.
The following inequalities arise from counting arguments on the number of edges joining $U$ with $D_0$ and $U$ with $D\setminus D_0$, respectively.
$$\sum_{v\in D_0} \deg_{D'}(v)\geq \sum_{i=1}^{k-1} \deg_{D_0}(u_i) \hspace{.3cm} \text{ and } \hspace{.3cm} |D\setminus D_0|(k-1)\geq \sum_{i=1}^{k-1} \deg_{D\setminus D_0}(u_i).$$ 
By the previous inequalities  and the fact that $D$ is a $k$-dominating set of $G$, we deduce that
\begin{align*}
\sum_{v\in D_0} \deg_{D'}(v)+|D\setminus D_0|(k-1)&\geq \sum_{i=1}^{k-1} \deg_{D_0}(u_i)+ \sum_{i=1}^{k-1} \deg_{D\setminus D_0}(u_i) \\ 
             &= \sum_{i=1}^{k-1} \deg_{D}(u_i)\\
&\geq k(k-1).
\end{align*}
Now, we define $D''\subseteq V(G)$ as a set of minimum cardinality among all supersets $W$ of $D'$ such that $\deg_W(x)\geq k-1$ for every vertex $x\in D$.
Since $\deg_{D'}(x)\geq k-1$ for
every $x\in D\setminus D_0$, the condition on $W$ is equivalent to that every vertex $v\in D_0$ has at least $k-1-\deg_{D'}(v)$ neighbours in $W\setminus D$. Hence, by the minimality of $D''$ and the inequality chain above, we deduce that 
\begin{align*}
|D''\setminus D'|&\leq |D_0|(k-1)-\sum_{v\in D_0} \deg_{D'}(v)\\
&=|D|(k-1)-\left(\sum_{v\in D_0} \deg_{D'}(v)+|D\setminus D_0|(k-1)\right)\\
&\leq |D|(k-1)-k(k-1).
\end{align*}  
Moreover, it is easy to check that $D''$ is a $k$-tuple dominating set of $G$ because each vertex in $V(G)\setminus D$ is dominated $k$ times by vertices of $D\subseteq D''$ (recall that $D$ is a $k$-dominating set of $G$) and  the construction of $D''$ ensures that each vertex in $D$ is dominated $k$ times by vertices of $D''$. Hence, 
\begin{align*}
\gamma_{\times k}(G)\leq |D''|& =|D'|+|D''\setminus D'|\\
                   &\leq |D|+ k-1 +|D|(k-1)-k(k-1)\\
                   &=k\gamma_{k}(G)-(k-1)^2,
\end{align*}
which completes the proof.
\end{proof}

The bound above is tight. For instance, it is achieved by any complete bipartite graph $K_{k,k'}$ with $k'\geq k$, as $\gamma_{\times k}(K_{k,k'})=2k-1$ and $\gamma_{k}(K_{k,k'})=k$.
When $k=2$, Theorem~\ref{teo-open-problem} leads to the  relationship $\gamma_{\times 2}(G)\leq 2\gamma_2(G)-1$ given in 2018 by Bonomo et al. \cite{BBGMS2018}.

\vspace{.15cm}
%---------------------------------------------------------
A set $D\subseteq V(G)$ is a $2$-packing of a graph $G$ if $N[u]\cap N[v]=\emptyset$ for every pair of different vertices $u,v\in D$. The $2$-packing number of $G$, denoted by $\rho(G)$,  is the maximum cardinality among all $2$-packings of $G$.

The next theorem relates the $k$-tuple domination number with the $2$-packing number of a graph. Note that the bounds given in this result are generalizations of the bounds $\gamma_{\times 2}(G)\geq 2\rho(G)$ due to Chellali et al. \cite{MR2232995}, and  $\gamma_{\times 2}(G)\leq |V(G)|-\rho(G)$ due to Chellali and Haynes \cite{MR2137931}.

\begin{theorem}\label{teo-xk-packing}
Let $k\geq 2$ be an integer. For any graph $G$ of order $n$   and $\delta(G)\geq k$,
$$k\rho(G)\leq \gamma_{\times k}(G)\leq n-\rho(G) .$$
\end{theorem}

\begin{proof}
Let $D$ be a $\rho(G)$-set and $S$  a $\gamma_{\times k}(G)$-set. Since $\deg_S(v)\geq k$ for every $v\in D\setminus S$, and $\deg_S(v)\geq k-1$ for every $v\in D\cap S$, we deduce that
$$\gamma_{\times k}(G)=|S|\geq \sum_{v\in D\setminus S}\deg_S(v)+\sum_{v\in D\cap S}(\deg_S(v)+1)\geq k|D|=k\rho(G),$$
and the lower bound follows.

Next, let us proceed to prove that $V(G)\setminus D$ is a $k$-tuple dominating set of $G$. Since $\delta(G)\geq k$, $N(D)\cap D=\emptyset$ and $\deg_{D}(x)\leq 1$ for every $x\in V(G)\setminus D$, we deduce that $\deg_{V(G)\setminus D}(v)\geq k$ for every $v\in D$ and $\deg_{V(G)\setminus D}(v)\geq k-1$ for every  $v\in V(G)\setminus D$. Hence, $V(G)\setminus D$ is a $k$-tuple dominating set of $G$, as desired.

Therefore,  $\gamma_{\times k}(G)\leq |V(G)\setminus D|=n-\rho(G)$, which completes the proof.
\end{proof}

Let $\mathcal{H}$ be the family of graphs $H_{k,r}$ defined as follows. For any pair of integers $k,r\in \mathbb{Z}$, with $k\ge 2$ and $r\ge 1$, the graph $H_{k,r}$ is obtained from a complete graph $K_{kr}$ and an empty graph $rK_1$ such that   $V(H_{k,r})=V(K_{kr})\cup V(rK_1)$,  $V(K_{kr})=\{v_1,\dots, v_{kr}\}$ and $V(rK_1)=\{u_1,\dots, u_r\}$ and $E(H_{k,r})=E(K_{kr})\cup (\bigcup_{i=0}^{r-1}\{u_{i+1}v_{ki+1},\dots, {u_{i+1}}v_{ki+k}\}).$  Figure~\ref{figure-packing} shows a graph of this family.	
Observe that $|V(H_{k,r})|=r(k+1)$, $\gamma_{\times k}(H_{k,r})=kr$ and $\rho(H_{k,r})=r$  for every $H_{k,r}\in \mathcal{H}$. Therefore, for these graphs the bounds given in Theorem \ref{teo-xk-packing} are tight, i.e.,  $\gamma_{\times k}(H_{k,r})=k\rho(H_{k,r})=|V(H_{k,r})|-\rho(H_{k,r})$.
  
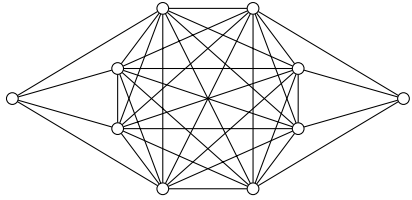
\begin{figure}[ht]
\centering
\begin{tikzpicture}[scale=.4, transform shape]
\node [draw, shape=circle] (v1) at  (0,6) {};
\node [draw, shape=circle] (v2) at  (-1.5,4) {};
\node [draw, shape=circle] (v3) at  (-1.5,2) {};
\node [draw, shape=circle] (v4) at  (0,0) {};

\node [draw, shape=circle] (v5) at  (3,0) {};
\node [draw, shape=circle] (v6) at  (4.5,2) {};
\node [draw, shape=circle] (v7) at  (4.5,4) {};
\node [draw, shape=circle] (v8) at  (3,6) {};

\node [draw, shape=circle] (u1) at  (-5,3) {};
\node [draw, shape=circle] (u2) at  (8,3) {};

\draw (v2)--(v1)--(u1)--(v2)--(v3)--(u1)--(v4)--(v3);
\draw (v6)--(v5)--(u2)--(v6)--(v7)--(u2)--(v8)--(v7);
\draw (v4)--(v5)--(v7)--(v1)--(v3)--(v5)--(v8)--(v1);
\draw (v1)--(v4)--(v6)--(v8)--(v2)--(v5)--(v1)--(v6);
\draw (v2)--(v4)--(v7)--(v2)--(v6)--(v3)--(v7);
\draw (v3)--(v8)--(v4);
\end{tikzpicture}
\caption{The graph $H_{4,2}\in \mathcal{H}$.}\label{figure-packing}
\end{figure}  

%---------------------------------------------------------------

%---------------------------------------------------------

In \cite{Harary2000}, Harary and Haynes showed that $\gamma_{\times k}(G)\geq \frac{2kn-2m}{k+1}$ for any graph $G$ of order $n$ and size $m$ with $\delta(G)\geq k-1$.
The next result is a partial refinement of the bound above because it only considers graphs with minimum degree at least $k$.

\begin{proposition}\label{teo-final}
Let $k\geq 2$ be an integer. For any graph $G$ of order $n$ and size $m$ with $\delta(G)\geq k$,   
$$\gamma_{\times k}(G)\geq \frac{(\delta(G)+k)n-2m}{\delta(G)+1}.$$
\end{proposition}

\begin{proof}
Let $S$ be a $\gamma_{\times k}(G)$-set and $\overline{S}=V(G)\setminus S$. Hence,
\begin{align*}
2m &=\sum_{v\in S}\deg_{S}(v)+2\sum_{v\in \overline{S}}\deg_{S}(v)+\sum_{v\in \overline{S}}\deg_{\overline{S}}(v)\\
&=\sum_{v\in S}\deg_{S}(v)+\sum_{v\in \overline{S}}\deg_{S}(v)+\sum_{v\in \overline{S}}\deg_{V(G)}(v)\\
&\geq(k-1)|S|+k(n-|S|)+\delta(G)(n-|S|)\\
&=(k-1)|S|+(\delta(G)+k)(n-|S|)\\
&=(\delta(G)+k)n-(\delta(G)+1)|S|,
\end{align*}
which implies that $|S|\geq \frac{(\delta(G)+k)n-2m}{\delta(G)+1}.$ Therefore, the proof is complete.
\end{proof}

The bound above is tight. For instance, it is achieved for the join graph $G=K_k+C_k$ obtained from the complete graph $K_k$ and the cycle graph $C_k$, with $k\geq 3$. For this case, we have that $\gamma_{\times k}(G)=k$, $|V(G)|=2k$, $\delta(G)=k+2$ and $2|E(G)|=3k^2+k$. Also, it is achieved for the complete graph $K_n$ ($n\geq 3$) and any $k\in \{2, \ldots, n-1\}$.

\end{document}